\documentclass[a4paper,12pt, reqno]{amsart}

\usepackage{amsmath, amsthm, amscd, amsfonts, amssymb, graphicx, color}
\usepackage[bookmarksnumbered, colorlinks, plainpages]{hyperref}

\usepackage{enumitem}
\usepackage{url}
\usepackage{multirow}
\usepackage{graphicx}
\usepackage{subfigure}
\usepackage[pass]{geometry}
\usepackage{cite}
\usepackage{cancel}
\usepackage{color}
\usepackage{pgf,tikz}
\usetikzlibrary{arrows}
\usepackage{lscape}%{pdflscape}
\tikzstyle{v} = [circle, draw, inner sep=2pt, minimum size=3pt, fill=black]
\usetikzlibrary{calc,shapes,decorations.pathreplacing,matrix}
\usetikzlibrary{patterns}

\tikzset{square matrix/.style={
    matrix of nodes,
    column sep=-\pgflinewidth, row sep=-\pgflinewidth,
    nodes={draw,
      minimum height=4.5pt,
      anchor=center,
      text width=4.5pt,
      align=center,
      inner sep=0pt
    },
  },
  square matrix/.default=1.2cm
}

\newtheorem{Theorem}{Theorem}[section]
\newtheorem{Definition}[Theorem]{Definition}

\newtheorem{Lemma}[Theorem]{Lemma}
\newtheorem{Proposition}[Theorem]{Proposition}
\newtheorem{Corollary}[Theorem]{Corollary}
\newtheorem{Remark}[Theorem]{Remark}
\newtheorem{Example}[Theorem]{Example}

\DeclareMathOperator{\leaf}{leaf}
\DeclareMathOperator{\diam}{diam}

\begin{document}

\title{Total domination number of middle graphs}

\author[F. Kazemnejad]{Farshad Kazemnejad}
\address{Farshad Kazemnejad, Faculty of Basic Sciences, Department of Mathematics, Ilam University, P.O.Box 69315-516, Ilam, Iran.}
\email{kazemnejad.farshad@gmail.com}
\author[B. Pahlavsay]{Behnaz Pahlavsay}
\address{Behnaz Pahlavsay, Department of Mathematics, Hokkaido University, Kita 10, Nishi 8, Kita-Ku, Sapporo 060-0810, Japan.}
\email{pahlavsayb@gmail.com}
\author[E. Palezzato]{Elisa Palezzato}
\address{Elisa Palezzato, Department of Mathematics, Hokkaido University, Kita 10, Nishi 8, Kita-Ku, Sapporo 060-0810, Japan.}
\email{palezzato@math.sci.hokudai.ac.jp}
\author[M. Torielli]{Michele Torielli}
\address{Michele Torielli, Department of Mathematics, GI-CoRE GSB, Hokkaido University, Kita 10, Nishi 8, Kita-Ku, Sapporo 060-0810, Japan.}
\email{torielli@math.sci.hokudai.ac.jp}

\date{\today}

\begin{abstract}
A \emph{total dominating set} of a graph $G$ with no isolated vertices is a subset $S$ of the vertex set such that every vertex of $G$ is adjacent to a vertex in $S$. The \emph{total domination number} of $G$ is the minimum cardinality of a total dominating set of $G$. 
In this paper, we study the total domination number of middle graphs. Indeed, we obtain tight bounds for this number in terms of the order of the graph $G$. We also compute the total domination number of the middle graph of some known families of graphs explicitly. Moreover, some Nordhaus-Gaddum-like relations are presented for the total domination number of middle graphs.
\\[0.2em]

\\[0.2em]

\noindent 
\end{abstract}
\maketitle
%--------------------------------------------------------------------------------
\section{Introduction}

The concept of total domination in graph theory was first introduced by Cockayne, Dawes and Hedetniemi in \cite{CDH} and it has been studied extensively by many researchers in the last years, see for example \cite{HHS5}, \cite{HHS6}, \cite{HeYe13}, \cite{3totdominrook}, \cite{Kaz19}, \cite{totcolordominmiddle}, 
\cite{romandomin} and \cite{dominLatin}. The literature on this subject has been surveyed and detailed in the recent book~\cite{HeYe13}. %In this paper, we study the total domination number of central graphs. In the sequel we remind some concepts and terminology which are used in this paper. 
We refer to \cite{bondy2008graph} as a general reference on graph theory.

Let $G$ be a graph with the vertex set $V(G)$ of \emph{order}
$n$ and the edge set $E(G)$ of \emph{size} $m$.
The \emph{open neighborhood} and the \emph{closed neighborhood} of a
vertex $v\in V(G)$ are $N_{G}(v)=\{u\in V(G)~|~ uv\in E(G)\}$ and
$N_{G}[v]=N_{G}(v)\cup \{v\}$, respectively. For a connected graph $G$, the \emph{degree} of a
vertex $v$ is defined as $d_G(v)=\vert N_{G}(v) \vert $. 
The \emph{distance} $d_G(v,w)$ in $G$ of two vertices $v,w\in V(G)$ is the length of the shortest path connecting $v$ and $w$.
The \emph{diameter} $\diam(G)$ of $G$ is the shortest distance between any two vertices in $G$.
%\begin{Definition} 
A \emph{dominating set} of a graph $G$ is a set $S\subseteq V(G)$ such that  $N_G[v]\cap
S\neq \emptyset$, for any vertex $v\in V(G)$. The \emph{domination number} of $G$ is the minimum cardinality of a dominating set of $G$ and is denoted by $\gamma(G)$. %Moreover, a dominating set of $G$ of cardinality $\gamma(G)$ is called a $\gamma$-set of $G$.
%\end{Definition}

\begin{Definition} Let $G$ be a graph with no isolated vertices. A \emph{total dominating set} of $G$ is a set $S\subseteq V(G)$
such that  $N_G(v)\cap
S\neq \emptyset$, for any vertex $v\in V(G)$. The \emph{total domination number} of $G$ is the minimum cardinality of a total dominating set of $G$ and is denoted by $\gamma_t(G)$. %Moreover, a total dominating set of $G$ of cardinality $\gamma_t(G)$ is called a $\gamma_t$-set of $G$.
\end{Definition}
%The\emph{minimum} and \emph{maximum degree} of a vertex in $G$ are denoted by $\delta =\delta (G)$ and $\Delta =\Delta (G)$, respectively. 
%We write $K_{n}$, $C_{n}$ and $P_{n}$ for a \emph{complete} graph, a \emph{cycle} graph and a \emph{path} graph of order $n$, respectively, while $W_n$ and $K_{n_1,n_2,\cdots,n_p}$, denote a \emph{wheel} graph of order $n+1$, and a \emph{complete $p$-partite} graph, respectively. The \emph{friendship} graph $F_n$ is obtained by joining $n$ copies of the cycle graph $C_3$ with a common vertex.
\begin{Example} Consider the path $P_3$ with vertex set $\{v_1,v_2,v_3\}$ and edge set $\{v_1v_2,v_2v_3\}$. Then the set $S=\{v_1,v_2\}$ is a total dominating set of $P_3$.
\end{Example}
For any non-empty $S\subseteq V(G)$, we denote by $G[S]$ the subgraph of $G$ induced on $S$. For any $v\in V(G)$, we denote by $G\setminus v$ the subgraph of $G$ induced on $V(G)\setminus \{v\}$. 

The \emph{complement} $\overline{G}$ of $G$ is a graph with vertex set $V(G)$ such that for every two vertices $v$ and $w$, $vw\in E(\overline{G})$ if and only if $vw\not\in E(G)$. 

%A \emph{vertex cover} of the graph $G$ is a set $D \subseteq V(G)$ such that every edge of $G$ is incident to at least one element of $D$. The \emph{vertex cover number} of $G$, denoted by $\tau(G)$, is the minimum cardinality of a vertex cover of $G$. Moreover, an \emph{edge cover} of $G$ is a set $S \subseteq E(G)$ such that every vertex of $G$ is incident to at least one edge in $S$. The \emph{edge cover number} of $G$, denoted by $\rho(G)$, is the minimum cardinality of an edge cover of $G$. An \emph{independent set} of $G$ is a subset of vertices of $G$, no two of which are adjacent. Also a \emph{maximum independent set} is an independent set of the largest cardinality in $G$. This cardinality is called the \emph{independence number} of $G$, and is denoted by $\alpha(G)$. The \emph{clique number} of $G$ is the maximum cardinality of the vertex set of a clique in $G$. For a tree graph $G$, any vertex of degree one is called a \emph{leaf} and the neighbour of a leaf is called a \emph{support vertex} of $G$. 
The \emph{line graph} of $G$, denoted by $L(G)$, is the graph with vertex set $E(G)$, where vertices $x$ and $y$ are adjacent in $L(G)$ if and only if edges $x$ and $y$ share a common vertex in $G$.

In \cite{HamYos}, the authors introduced the notion of the middle graph $M(G)$ of $G$ as an intersection graph on $V(G)$.

\begin{Definition}
 The \emph{middle graph} $M(G)$ of a graph $G=(V,E)$ is the graph whose vertex set is $V(G)\cup E(G)$ and two vertices $x, y$ in the vertex set of $M(G)$ are adjacent in $M(G)$ in case one the following holds: 
 \begin{enumerate} 
 \item $x, y$ are in $E(G)$ and $x, y$ are adjacent in $G$.
 \item $x$ is in $V (G)$, $y$ is in $E(G)$, and $x, y$ are incident in $G$. 
 \end{enumerate}
 \end{Definition}
 
 \begin{Example} Consider the graph $P_3$, then the middle graph $M(P_3)$ is the one in Figure~\ref{fig:midlepaths}.
 \end{Example}
 
 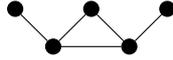
\begin{figure}[th]
\centering
\begin{tikzpicture}
\tikzstyle{v} = [circle, draw, inner sep=2pt, minimum size=3pt, fill=black]
\draw (-1,0) node[v](1){}; 
\draw (-0.5,-0.5) node[v](2){}; 
\draw (0,0) node[v](3){};
\draw (0.5,-0.5)  node[v](4){};
\draw (1,0)  node[v](5){};
\draw (1)--(2)--(3)--(4)--(5);
\draw (2)--(4);
\end{tikzpicture}
\caption{The middle graph $M(P_3)$ }\label{fig:midlepaths}
\end{figure}
 
It is obvious that $M(G)$ contains the line graph $L(G)$ as induced subgraph, and that if $G$ is a graph of order $n$ and size $m$, then $M(G)$ is a graph of order $n+m$ and size $2m+|E(L(G))| $ which is obtained by subdividing each edge of $G$ exactly once and joining all the adjacent edges of $G$ in $ M(G)$. 

In order to avoid confusion throughout the paper, we fix a ``standard'' notation for the vertex set and the edge set of $M(G)$. Assume $V(G)=\{v_1,v_2,\dots, v_n\}$, then we set $V(M(G))=V(G)\cup \mathcal{M}$, where $\mathcal{M}=\{m_{ij}~|~ v_iv_j\in E(G)\}$ and $E(M(G))=\{v_im_{ij},v_jm_{ij}~|~ v_iv_j\in E(G)\}\cup E(L(G)) $.

%----------------------------------------------------------------

%\vspace{0.2cm}
%For standard graph theory terminology not given here we refer to \cite{West}. 

In this article, we continue our study from \cite{dominmiddle} on domination of middle graphs. The paper proceeds as follows. In Section 2, we describe explicitly the total domination number of the middle graph of several known families of graphs and we present some upper and lower bounds for $\gamma_t(M(G))$ in terms of the order of the graph $G$. In Section 3, we describe bounds for the total domination number of the middle graph of trees. In Section 4, we obtain the same type of results for $\gamma_t(M(G\circ K_1))$, $\gamma_t(M(G\circ P_2))$ and $\gamma_t(M(G+ K_p))$. Finally, in the last Section, we present some Nordhaus-Gaddum like relations for the total domination number of middle graphs.

\section{Middle graph of known graphs and their total domination number}\label{t80}
%In this section, we obtain the total domination number of the middle graph of some special families of graphs. %The total domination number of the middle graph of cycles and paths are given in the first two propositions.
We start our study on total domination with two key Lemmas.

%A similar proof to the one of Lemma~\ref{lemma:dominationisalledges} gives us the following result.
\begin{Lemma}\label{lemma:totdominationisalledges}
Let $G$ be a connected graph of order $n\ge3$ and $S$ a total dominating set of $M(G)$. Then there exists $S'\subseteq E(G)$ a total dominating set of $M(G)$ with $|S'|\le|S|$.
\end{Lemma}
\begin{proof} If $S\subseteq E(G)$, then take $S'=S$. On the other hand, assume that there exists $v\in S\cap V(G)$. If all edges adjacent to $v$ are already in $S$, then take $S_1=S\setminus\{v\}$. Otherwise, let $e\in E(G)\setminus S$ be an edge adjacent to $v$. Then consider $S_1=(S\cup\{e\})\setminus\{v\}$. Since $S$ is finite, then this process terminates after a finite number of steps, and hence we obtain the described $S'$.
\end{proof}

\begin{Lemma}\label{lemma:totdominationdeletionvertex}
Let $G$ be a connected graph of order $n\ge2$ and $v\in V(G)$ a vertex not adjacent to any vertex of degree $1$. Then
$$\gamma_t(M(G\setminus v))\le \gamma_t(M(G)) \le \gamma_t(M(G\setminus v))+1.$$ 
\end{Lemma}
\begin{proof} Let $S$ be a total dominating set of $M(G\setminus v)$. This implies that for every $w\in N_G(v)$, $w\in S$ or there exists an edge of the form $ww_0\in E(G\setminus v)$ such that $ww_0\in S$. As a consequence,  $S\cup\{vw\}$ is a total dominating set of $M(G)$, for any $w\in N_G(v)$, and hence $\gamma_t(M(G)) \le \gamma_t(M(G\setminus v))+1$.

On the other hand, let $S$ be a minimal total dominating set of $M(G)$. By Lemma~\ref{lemma:totdominationisalledges}, we can assume that $S\subseteq E(G)$. 
Consider $S_v=N_{M(G)}(v)\cap S$. Since $S$ is a minimal total dominating set, $|S_v|\ge 1$. Assume $S_v=\{e_1,\dots, e_k\}$. For any $1\le i\le k$, $e_i$ is an edge of $G$ of the form $w_iv$. By the assumption on $v$, $N_{M(G)}(w_1)=\{e_1,e_{11},\dots, e_{1p}\}$ with $p\ge1$, for some $e_{1j}\in E(G\setminus v)$. If $S\cap\{e_{11},\dots, e_{1p}\}\ne\emptyset$, then consider $S_1=(S\setminus e_1)\cup\{w_1\}$, otherwise $S_1=(S\setminus e_1)\cup\{e_{11}\}$. By applying the same construction for each $e_i$, we obtain $S_k$ a total dominating set of $M(G\setminus v)$ with $|S_k|=|S|$, and hence $\gamma_t(M(G\setminus v))\le \gamma_t(M(G))$.
\end{proof}

We are now ready to describe explicitly the total dominating number of the middle graph of several known families of graphs.

\begin{Proposition}\label{prop:mintotdominstar} 
For any star graph $K_{1,n}$ on $n+1$ vertices, with $n\ge 2$, we have  
$$\gamma_t(M(K_{1,n}))=n.$$
\end{Proposition}
\begin{proof} To fix the notation, assume $V(K_{1,n})=\{v_0,v_1,\dots, v_n\}$ and $E(K_{1,n})=\{v_0v_1,v_0v_2,\dots, v_0v_n\}$. Then $V(M(K_{1,n}))=V(K_{1,n})\cup \mathcal{M}$, where $\mathcal{M}=\{ m_i~|~1\leq i \leq n \}$.

If $S=\mathcal{M}$, then $S$ is a total dominating set of $M(K_{1,n})$ with $|S|=n$, and hence $\gamma_t(M(K_{1,n}))\le n$. On the other hand, using \cite[Proposition 3.1]{dominmiddle}, $n=\gamma(M(K_{1,n}))\le \gamma_t(M(K_{1,n}))$.
\end{proof}

\begin{Definition}
A double star graph $S_{1,n,n}$ is obtained from the star graph $K_{1,n}$ by replacing every edge with a path of length $2$. 
\end{Definition}

\begin{Proposition}\label{prop:mintotdomindoublestar}
For any double star graph $S_{1,n,n}$ on $2n+1$ vertices, with $n\ge 1$, we have $$\gamma_t(M(S_{1,n,n}))=2n.$$
\end{Proposition}
\begin{proof} 
To fix the notation, assume $V(S_{1,n,n})=\{v_0,v_1,\dots, v_{2n}\}$ and $E(S_{1,n,n})=\{v_0v_i,v_iv_{n+i}~|~1\leq i \leq n\}$. Then $V(M(S_{1,n,n}))=V(S_{1,n,n})\cup \mathcal{M}$, where $\mathcal{M}=\{ m_{i},m_{i(n+i)}~|~1\leq i \leq n \}$. %and $E(M(S_{1,n,n}))=\{v_0m_{0i},m_{0i}v_i, v_im_{i(n+i)}, m_{i(n+i)}v_{n+i}~|~1\leq i \leq n \}\cup \{m_{ij}m_{ik}~|~v_jv_k\in E(S_{1,n,n}) \}$. 

If $S=\mathcal{M}$, then $S$ is a total dominating set of $M(S_{1,n,n})$ with $|S|=2n$, and hence $\gamma_t(M(S_{1,n,n}))\le 2n$. 

On the other hand, let $S$ be a total dominating set $M(S_{1,n,n})$. By Lemma~\ref{lemma:totdominationisalledges}, we can assume that $S\subseteq\mathcal{M}$. Since, for every $1\le i\le n$, $N_{M(S_{1,n,n})}(v_{n+i})=\{m_{i(n+i)}\}$, then $m_{i(n+i)}\in S$ for every $1\le i\le n$. Similarly, for every $1\le i\le n$, $N_{M(S_{1,n,n})}(m_{i(n+i)})=\{m_i,v_i,v_{n+i}\}$ implies that $m_i\in S$ for every $1\le i\le n$, and hence $\mathcal{M}\subseteq S$. This implies that $\gamma_t(M(S_{1,n,n}))\ge2n$.
\end{proof}

\begin{Proposition}\label{prop:mintotdominpath}
\label{gamma_{t}(M(P_n))}
For any path $P_n$ of order $n\geq 3$, $$\gamma_{t}(M(P_n))=\lceil \frac{2n}{3} \rceil.$$
\end{Proposition}
\begin{proof}
To fix the notation, assume $V(P_n)=\{v_1,\dots, v_n\}$ and $E(P_n)=\{v_iv_{i+1}~|~1\le i\le n-1\}$. Then $V(M(P_n))=V\cup \mathcal{M}$ where $V=V(P_n)$ and $\mathcal{M}=\{ m_{i(i+1)}~|~1\leq i \leq n-1 \}$. 

If $n \equiv 0 \mod 3$, then consider 
$$S=\{m_{12},m_{23},m_{45},m_{56},\dots, m_{(n-2)(n-1)}, m_{(n-1)n}\}.$$ 
We have that $S$ is a total dominating set of $M(P_n)$ with $|S|=\frac{2n}{3}$. 
If $n \equiv 1 \mod 3$, then consider 
$$S=\{m_{12},m_{23},m_{45},m_{56},\dots, m_{(n-3)(n-2)},m_{(n-2)(n-1)}\}\cup\{m_{(n-1)n}\}.$$ We have that $S$ is a total dominating set of $M(P_n)$ with $|S|=\lceil\frac{2n}{3}\rceil$. 
If $n \equiv 2 \mod 3$, then consider 
$$S=\{m_{12},m_{23},m_{45},m_{56},\dots, m_{(n-4)(n-3)},m_{(n-3)(n-2)}\}\cup\{m_{(n-2)(n-1)},m_{(n-1)n}\}.$$ 
We have that $S$ is a total dominating set of $M(P_n)$ with $|S|=\lceil\frac{2n}{3}\rceil$.
This implies $\gamma_{t}(M(P_n))\le\lceil \frac{2n}{3} \rceil.$

On the other hand, let $S$ be a total dominating set for $M(P_n)$. For every $i=1,\dots, n-2$, let $G_i=P_n[v_i,v_{i+1},v_{i+2}]$. Since $S$ dominates all vertices of the graph $M(G_i)$, $|S\cap V(M(G_i))|\ge 2$. This implies that $|S|\ge \lceil\frac{2n}{3}\rceil$.
\end{proof}

Since if we delete a vertex from a complete graph $K_{n+1}$ we obtain a graph isomorphic to $K_n$, Lemma~\ref{lemma:totdominationdeletionvertex} gives us the following result.

\begin{Lemma}\label{lemma:ineqtotdomincomplete}
For any $n\ge3$, we have
$$\gamma_t(M(K_n))\le\gamma_t(M(K_{n+1}))\le\gamma_t(M(K_n))+1.$$
\end{Lemma}
%\begin{proof} Assume $V(K_n)=\{v_1,\dots, v_n\}$ and $V(K_{n+1})=V(K_n)\cup\{v_{n+1}\}$. Then $V(M(K_n))=V(K_n)\cup \mathcal{M}_n$, where $\mathcal{M}_=\{ m_{ij}~|~1\le i<j\le n\}$ and $V(M(K_{n+1}))=V(M(K_n))\cup\{ m_{i(n+1)}~|~1\le i\le n\}$.
%
%Let $S$ be a total dominating set of $M(K_{n+1})$. By Lemma~\ref{lemma:totdominationisalledges}, we can assume that $S\subseteq E(K_{n+1})$. Since $S$ is a total dominating set of $M(K_{n+1})$, then $S\cap\{m_{i(n+1)}~|~1\le i\le n\}\ne \emptyset$. Let $i_0=\min\{i~|~m_{i(n+1)}\in S\}$. Without loss of generality, we can assume that there exists $j_0=1,\dots, n$ with $j_0\ne i_0$ such that $m_{i_0j_0}\notin S$. Then consider $S_0=(S\setminus \{m_{i_0(n+1)}\})\cup\{m_{i_0j_0}\}$. Doing this process for any index in $\{i~|~m_{i(n+1)}\in S\}$, we obtain $S'$ a total dominating set of $M(K_{n})$ with $|S'|=|S|$. This implies that $\gamma_t(M(K_n))\le\gamma_t(M(K_{n+1}))$.
%
%On the other hand, let $S$ be a total dominating set of $M(K_{n})$. By Lemma~\ref{lemma:totdominationisalledges}, we can assume that $S\subseteq E(K_{n})$. This implies that if we consider $S$ as subset of $E(K_{n+1})$, then $S$ dominates all vertices of $M(K_{n+1})$ except for $v_{n+1}$. This implies that $S'=S\cup\{m_{1(n+1)}\}$ is a total dominating set of $M(K_{n+1})$ with $|S'|=|S|+1$. This implies that $\gamma_t(M(K_{n+1}))\le\gamma_t(M(K_n))+1$.
%\end{proof}

%Similarly to Proposition~\ref{prop:mindomincompletegr}, we obtain the following result.
\begin{Proposition}\label{prop:mintotdomincompletegr}
Let $K_n$ be the complete graph on $n\ge2$ vertices. Then 
$$\gamma_t(M(K_n))= \lceil \frac{2n}{3} \rceil
%\begin{cases}
%\lceil \frac{2(n-1)}{3}\rceil & \text{ if $n\equiv0\mod 3$}\\ 
%\lceil \frac{2(n-1)}{3}\rceil+1 & \text{ if $n\equiv1,2\mod 3$.} 
%\end{cases}
$$
\end{Proposition}
\begin{proof} 
If $2\le n\le 4$, a direct computation shows that $\gamma_t(M(K_n))= \lceil \frac{2n}{3}\rceil$.
%Similarly, if $n=3$, we have $\gamma_t(M(K_3))=2=\frac{2n}{3}= \lceil \frac{2n}{3}\rceil$.
Assume now $n\ge5$. The graph $K_n$ has several subgraphs isomorphic to $P_n$, and hence $M(K_n)$ has subgraphs isomorphic to $M(P_n)$. Fix one of those and consider $S$ a total dominating set of $M(P_n)$. Since $S$ is also a total dominating set for $M(K_n)$, we have $\gamma_t(M(K_n))\le \lceil \frac{2n}{3} \rceil$.

We will prove the opposite inequality by induction. 
%It is easy to see that we have equality when $2\le n\le 5$. 
Assume that we have equality for $\gamma_t(M(K_n))$ and we want to prove it for $\gamma_t(M(K_{n+1}))$. 
If $n\equiv2\mod 3$, then $n+1\equiv0\mod 3$, and hence, $\gamma_t(M(K_n))=\lceil\frac{2n}{3}\rceil=\lceil \frac{2(n+1)}{3}\rceil$. On the other hand, by Lemma~\ref{lemma:ineqtotdomincomplete}, $\gamma_t(M(K_n))\le\gamma_t(M(K_{n+1}))$. This fact, together with the first part of the proof, implies that $\gamma_t(M(K_{n+1}))=\lceil \frac{2(n+1)}{3}\rceil$. 
If $n\equiv0,1\mod 3$, by Lemma~\ref{lemma:ineqtotdomincomplete} and the first part of the proof, it is enough to show that $\gamma_t(M(K_n))<\gamma_t(M(K_{n+1}))$. As a contradiction, assume that $\gamma_t(M(K_n))=\gamma_t(M(K_{n+1}))$. If $n\equiv0\mod 3$, then $n-1\equiv2\mod 3$, and hence this would implies $\gamma_t(M(K_{n-1}))=\gamma_t(M(K_n))=\gamma_t(M(K_{n+1}))$. Similarly, if $n\equiv1\mod 3$, then $n+1\equiv2\mod 3$, and hence $\gamma_t(M(K_n))=\gamma_t(M(K_{n+1}))=\gamma_t(M(K_{n+2}))$. Hence we need to show that $\gamma_t(M(K_n))<\gamma_t(M(K_{n+2}))$, when $n\ge4$.
Let $S$ be a total dominating set of $M(K_{n+2})$. To fix the notation, assume $V(K_{n+2})=\{v_1,\dots, v_{n+2}\}$ and $V(M(K_{n+2}))=V(K_{n+2})\cup \mathcal{M}$, where $\mathcal{M}=\{ m_{ij}~|~1\le i<j\le n+2\}$. By Lemma~\ref{lemma:totdominationisalledges}, we can assume that $S\subseteq \mathcal{M}$. After possibly relabeling $V(K_{n+2})$, we can assume that $m_{(n+1)(n+2)}\in S$. Since $S$ is a total dominating set of $M(K_{n+2})$, then it contains at least one element of the form $m_{i(n+1)}$ or $m_{i(n+2)}$, for some $i=1,\dots, n$. By construction, $M(K_n)$ is isomorphic to $M(K_{n+2}[v_1,\dots, v_n])$, this implies that, similarly to the proof of Lemma~\ref{lemma:ineqtotdomincomplete}, we can construct $S'$ a total dominating set of $M(K_n)$ by exchanging a vertex of the form $m_{i(n+1)}$ or $m_{i(n+2)}$ with one of the form $m_{ij}$ and just discarding $m_{(n+1)(n+2)}$. This implies that $|S'|<|S|$, and hence $\gamma_t(M(K_n))<\gamma_t(M(K_{n+2}))$.
\end{proof}

\begin{Theorem}\label{theo:lowerboundtotdomin} Let $G$ be any graph of order $n$. Then
$$\lceil \frac{2n}{3} \rceil\le \gamma_t(M(G))\le n-1$$
\end{Theorem}
\begin{proof} From $G$ we can obtain  graph isomorphic to $K_n$ by adding all the necessary edges. This implies that we can see $G$ as a subgraph of $K_n$, and hence $M(G)$ as a subgraph of $M(K_n)$. Since any total dominating set of $M(G)$ is also a total dominating set for $M(K_n)$, this implies that $\gamma_t(M(G))\ge\gamma_t(M(K_n))$. We obtain the left inequality by Proposition~\ref{prop:mintotdomincompletegr}.

 Let $T$ be a spanning tree of $G$ and $S$ a minimal total dominating set of $M(T)$. By Lemma~\ref{lemma:totdominationisalledges}, we can assume that $S\subseteq E(T)$. This implies that $|S|\le|E(T)|=n-1$. By construction, $S$ is also a total dominating set of $M(G)$, and hence $\gamma_t(M(G))\le n-1$.
\end{proof}

\begin{Remark} By Propositions~\ref{prop:mintotdominstar} and \ref{prop:mintotdominpath}, the inequalities of Theorem~\ref{theo:lowerboundtotdomin} are all sharp.
\end{Remark}

\begin{Theorem}\label{prop:graphwithapath}
 If $G$ is a graph with order $n$ and there exists a subgraph of $G$ isomorphic to $P_n$, then  
 $$\gamma_t(M(G)) = \lceil \frac{2n}{3}\rceil. $$
\end{Theorem}
\begin{proof} Since $G$ has a subgraph isomorphic to $P_n$, then $M(G)$ has a subgraph isomorphic to $M(P_n)$. Moreover, any total dominating set of $M(P_n)$  is also a total dominating set for $M(G)$. By Proposition~\ref{prop:mintotdominpath}, this implies that $\gamma_t(M(K_n))\le \lceil \frac{2n}{3} \rceil$. We conclude by Theorem~\ref{theo:lowerboundtotdomin}.
\end{proof}

Directly from Theorem~\ref{prop:graphwithapath}, we obtain the following result.
\begin{Corollary}\label{corol:mintotdominfamily} For any $n\ge3$, $$\gamma_t(M(P_n))=\gamma_t(M(C_n))=\gamma_t(M(W_n))=\gamma_t(M(K_n))=\lceil \frac{2n}{3}\rceil.$$
\end{Corollary}

\begin{Definition} The \emph{friendship} graph $F_n$ of order $2n+1$ is obtained by joining $n$ copies of the cycle graph $C_3$ with a common vertex.
\end{Definition}

\begin{Proposition}\label{prop:mintotdominfriendship}
Let $F_n$ be the friendship graph with $n\ge2$. Then $$\gamma_t(M(F_n))= 2n.$$
\end{Proposition}
\begin{proof} To fix the notation, assume $V(F_n)=\{v_0,v_1,\dots, v_{2n}\}$ and $E(F_n)=\{v_0v_1,v_0v_2,\dots, v_0v_{2n}\}\cup\{v_1v_2, v_3v_4,\dots,v_{2n-1}v_{2n}\}$. Then $V(M(F_n))=V(F_n)\cup \mathcal{M}$, where $\mathcal{M}=\{ m_i~|~1\leq i \leq 2n \}\cup\{ m_{i(i+1)}~|~1\leq i \leq 2n-1 \text{ and } i \text{ is odd}\}$. 

Consider $S=\{ m_{i(i+1)}~|~1\leq i \leq 2n-1 \text{ and } i \text{ is odd}\}\cup\{v_i~|~i\text{ is odd}\}$. Then $S$ is a total dominating set for $M(F_n)$ with $|S|=2n$, and hence, $\gamma_t(M(F_n))\le 2n$.

On the other hand, since $F_n$ is obtained by joining $n$ copies of $C_3$ at $v_0$, any total dominating set $S$ of $M(F_n)$ induces a total dominating set of $M(C_3)$ as subgraph of $M(F_n)$. By Corollary~\ref{corol:mintotdominfamily}, $\gamma_t(M(C_3))=2$. This fact together with the fact that any two distinct copies of $M(C_3)$ in $M(F_n)$ share only $v_0$ implies that $|S|\ge2n$. This implies that $\gamma_t(M(F_n))\ge 2n$.
\end{proof}

Using Theorem~\ref{theo:lowerboundtotdomin}, we can describe the total domination number of the middle graph of a complete bipartite graph.

\begin{Proposition}\label{prop:mintotdomincompletebipartitegr}
Let $K_{n_1,n_2}$ be the complete bipartite graph with $n_2\geq n_1 \geq 2$. Then 
$$\gamma_t(M(K_{n_1,n_2}))= 
\begin{cases}
n_2+\lceil\frac{2n_1-n_2}{3}\rceil & \text{ if $n_1\le n_2\le 2n_1-1$}\\ 
n_2 & \text{ if $n_2\ge 2n_1$.} 
\end{cases}$$
\end{Proposition}
\begin{proof}
Assume $V(K_{n_1,n_2})=\{v_1,\dots, v_{n_1},u_1,\dots, u_{n_2}\}$ and $E(K_{n_1,n_2})=\{v_iu_j ~|~1\leq i \leq n_1, 1\leq j \leq n_2\}$. Then we have $V(M(K_{n_1,n_2}))=V(K_{n_1,n_2})\cup \mathcal{M}$, where
 $\mathcal{M}=\{ m_{ij}~|~1\leq i \leq n_1, 1\leq j \leq n_2 \}$.
 
 Assume first $n_1=n_2$. If $n_1\equiv0\mod 3$, then consider $$S=\{m_{11},m_{12},m_{23},m_{33},\dots,m_{(n_1-1)n_1},m_{n_1n_1}\}.$$ By construction, $S$ is a total dominating set of $M(K_{n_1,n_2})$ and $|S|=n_1+\frac{n_1}{3}=n_2+\frac{n_1}{3}=n_2+\lceil\frac{2n_1-n_2}{3}\rceil$.
 If $n_1\equiv1\mod 3$, then consider $$S=\{m_{11},m_{12},m_{23},m_{33},\dots,m_{(n_1-2)(n_1-1)},m_{(n_1-1)(n_1-1)}\}\cup\{m_{n_1(n_1-1)},m_{n_1n_1}\}.$$ By construction, $S$ is a total dominating set of $M(K_{n_1,n_2})$ and $|S|=n_1+\lceil\frac{n_1}{3}\rceil=n_2+\lceil\frac{n_1}{3}\rceil=n_2+\lceil\frac{2n_1-n_2}{3}\rceil$. 
 If $n_1\equiv2\mod 3$, then consider $$S=\{m_{11},m_{12},m_{23},m_{33}, \dots, m_{(n_1-1)(n_1-1)},m_{(n_1-1)n_1}\}\cup\{m_{n_1n_1}\}.$$ By construction, $S$ is a total dominating set of $M(K_{n_1,n_2})$ and $|S|=n_1+\lceil\frac{n_1}{3}\rceil=n_2+\lceil\frac{n_1}{3}\rceil=n_2+\lceil\frac{2n_1-n_2}{3}\rceil$.
 
Assume that $n_1+1\le n_2\le 2n_1-1$. Consider $$S'=\{m_{11},m_{1n_1+1},\dots, m_{(n_2-n_1)(n_2-n_1)},m_{(n_2-n_1)n_2}\}.$$ Let $G=K_{n_1,n_2}[u_{n_2-n_1+1},\dots,u_{n_1},v_{n_2-n_1+1},\dots,v_{n_1}]$. Then $G$ is isomorphic to a graph of the form $K_{n,n}$, where $n=2n_1-n_2$. This implies that by the first part of the proof, we can construct $S''$ a total dominating set of $M(G)$ with $|S''|=2n_1-n_2+\lceil\frac{2n_1-n_2}{3}\rceil$. Consider $S=S'\cup S''$. Then $S$ is a total dominating set of $M(K_{n_1,n_2})$ and $|S|=2(n_2-n_1)+2n_1-n_2+\lceil\frac{2n_1-n_2}{3}\rceil= n_2+\lceil\frac{2n_1-n_2}{3}\rceil$.

This implies that if $n_1\le n_2\le 2n_1-1$, then $\gamma_t(M(K_{n_1,n_2}))\le n_2+\lceil\frac{2n_1-n_2}{3}\rceil$.
 
 Assume now that $n_2\ge 2n_1$. Consider $$S=\{m_{11},m_{1n_1+1},\dots, m_{n_1n_1},m_{n_12n_1}\}\cup\{m_{n_12n_1+1},\dots,m_{n_1n_2}\},$$ then $S$ is a total dominating set of $M(K_{n_1,n_2})$ with $|S|=n_2$, and hence, $\gamma_t(M(K_{n_1,n_2}))\le n_2$.
 
On the other hand, assume first $n_1=n_2$. By Theorem~\ref{theo:lowerboundtotdomin}, we have $\gamma_t(M(K_{n_1,n_2}))\ge \lceil\frac{2(n_1+n_2)}{3}\rceil=n_2+\lceil\frac{2n_1-n_2}{3}\rceil.$
%Let $S$ be a total dominating set of $M(K_{n_1,n_2})$. By Lemma~\ref{lemma:dominationisalledges}, we can assume that $S\subseteq \mathcal{M}$. For every $i=1,\dots, n_1$, the vertex $u_i$ is dominated by an element of $S$. This implies that there are only two possibilities: there exist $1\le j<k\le n_2$ such that $m_{ij},m_{ik}\in S$ or there exist $1\le j\le n_2$ and $i\le k\le n_1$ with $i\ne k$ such that $m_{ij},m_{kj}\in S$. in Addition, since $n_1=n_2$, the two cases have to appear the same number of times. This implies that $\gamma_t(M(K_{n_1,n_2}))\ge n_2+\lceil\frac{n_1}{3}\rceil= n_2+\lceil\frac{2n_1-n_2}{3}\rceil.$

Assume that $n_1+1\le n_2\le 2n_1-1$. Let $S$ be a total dominating set of $M(K_{n_1,n_2})$. By Lemma~\ref{lemma:totdominationisalledges}, we can assume that $S\subseteq \mathcal{M}$. The construction of the first part of the proof is optimal since $S'$ and $S''$ have the smallest possible size by the argument discussed when $n_1=n_2$ and $n_2\ge 2n_1$.

This implies that if $n_1\le n_2\le 2n_1-1$, then $\gamma_t(M(K_{n_1,n_2}))= n_2+\lceil\frac{2n_1-n_2}{3}\rceil$.

Assume that $n_2\ge 2n_1$, then by \cite[Proposition 3.13]{dominmiddle}, we have $n_2=\gamma(M(K_{n_1,n_2}))\le \gamma_t(M(K_{n_1,n_2}))\le n_2$. This implies that $\gamma_t(M(K_{n_1,n_2}))=n_2$.

\end{proof}

\section{The middle graph of a tree}

Similarly to \cite[Proposition 2.4]{dominmiddle}, if we consider $T$ a tree and we denote by $\leaf(T)=\{v\in V(T)~|~d_T(v)=1\}$ the set of leaves of $T$, then we have the following result.
\begin{Proposition}\label{prop:mintotdomintreeleaf}
Let $T$ be a tree with $n\ge2$ vertices. Then $$\gamma_t(M(T))\ge |\leaf(T)|.$$
\end{Proposition}
\begin{proof} To fix the notation, assume $\leaf(T)=\{v_1,\dots,v_k\}$, for some $k\le n$. If $n=2$, then $T$ is isomorphic to $P_2$ and hence $\gamma_t(M(T))=2= |\leaf(T)|$. Assume that $n\ge3$ and let $S$ be a total dominating set of $M(T)$. Then, for each $i=1,\dots, k$, $S\cap N_{M(T)}[v_i]\ne\emptyset$. Since, if $i\ne j$, then $N_{M(T)}[v_j]\cap N_{M(T)}[v_i]=\emptyset$, we have that $|S|\ge k$. This implies that $\gamma_t(M(T))\ge k= |\leaf(T)|$.

\end{proof}

\begin{Remark}
Notice that by Proposition~\ref{prop:mintotdominstar}, the inequality described in Proposition~\ref{prop:mintotdomintreeleaf} is sharp.
\end{Remark}

It is sufficient to add some assumptions on the diameter of a tree $T$, to compute $\gamma_t(M(T))$ explicitly.

\begin{Theorem}\label{theomintotdomintreediam3}
Let $T$ be a tree of order $n\geq 4$ with $\diam(T)=3$. Then 
$$\gamma_t(M(T))=
\begin{cases}
n-2 & \text{ if there are two vertices with $d_T(v)\ge3$}\\ 
%n-1 & \text{ if there is one vertex with $d_T(v)\ge3$} \\
n-1 & \text{otherwise.}
\end{cases}$$
\end{Theorem}
\begin{proof} 
Since by assumption $\diam(T)=3$, then $T$ is a tree which is obtained by joining central vertex $v$ of $K_{1,p}$ and the central vertex $w$ of $K_{1,q}$ where $p+q=n-2$. Let $\leaf(T)=\{v_i~|~1\le i\le n-2\}$ be the set of leaves of $T$. Obviously $V(T)=\leaf(T)\cup \{v,w\}$ and $|\leaf(T)|=n-2$. Define $v_{n-1}=v$ and $v_{n}=w$.

Assume first that $p,q\ge2$, i.e. there are two vertices with $d_T(u)\ge3$. If we consider $S=\{m_{i(n-1)}~|~1\leq i \leq p \}\cup\{m_{in}~|~p+1\leq i \leq n-2\}$, then $S$ is a total dominating set of $M(T)$ with $|S|=n-2$, and hence $\gamma_t(M(T))\le n-2$. On the other hand, by Proposition~\ref{prop:mintotdomintreeleaf}, we have $\gamma_t(M(T))\ge n-2$.

Assume that $p\ge2$ and $q=1$, i.e. there is only one vertex with $d_T(u)\ge3$. Let $S$ be a total dominating set of $M(T)$. By Lemma~\ref{lemma:totdominationisalledges}, we can assume that $S\subseteq E(T)$. Since $N_{M(T)}(v_i)=\{m_{i(n-1)}\}$ for all $1\le i\le p=n-3$ and $N_{M(T)}(v_{n-2})=\{m_{(n-2)n}\}$, then $\{m_{i(n-1)}~|~1\le i\le p\}\cup\{m_{(n-2)n}\}\subseteq S$. Moreover, $N_{M(T)}(m_{(n-2)n})=\{m_{(n-1)n},v_n,v_{n-2}\}$ implies that $m_{(n-1)n}\in S$. This implies that $|S|\ge n-1$, and hence $\gamma_t(M(T))\ge n-1$.
On the other hand, by Theorem~\ref{theo:lowerboundtotdomin}, $\gamma_t(M(T))\le n-1$.

Assume that $p,q=1$, i.e. there are no vertices with $d_T(u)\ge3$. This implies that $T$ is isomorphic to $P_4$ and $n=4$, and hence by Proposition~\ref{prop:mintotdominpath}, $\gamma_t(M(T))=3=n-1$.
\end{proof}

In general, the opposite implication of Theorem~\ref{theomintotdomintreediam3} does not hold as the next example shows.
\begin{Example} Let $T$ be the tree in Figure~\ref{Fig:tree}. Then a direct computation shows that $\diam(T)=4$ and $\gamma_t(M(T))=5=n-2$.

\begin{figure}[th]
\centering
\begin{tikzpicture}
\tikzstyle{v} = [circle, draw, inner sep=2pt, minimum size=3pt, fill=black]
\draw (0,0) node[v](1){}; 
\draw (-0.5,1) node[v](2){}; 
\draw (0.5,1) node[v](3){};
\draw (-1,2)  node[v](4){};
\draw (-0.2,2)  node[v](5){};
\draw (0.2,2)   node[v](6){};  
\draw (1,2)   node[v](7){};  
\draw (1)--(2)--(4);
\draw (1)--(2)--(5);
\draw (1)--(3)--(6);
\draw (1)--(3)--(7);
\end{tikzpicture}
\caption{A tree on $7$ vertices}\label{Fig:tree}
\end{figure}
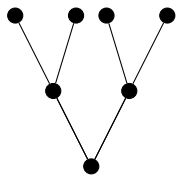

%\begin{figure}[th]
%\centering
%\begin{tikzpicture}
%\draw (0.5,0) node[v](1){}; 
%\draw (0.5,1) node[v](2){}; 
%\draw (0.5,2) node[v](3){};
%\draw (0,-1) node[v](6){};
%\draw (1,-1) node[v](7){};
%\draw (0,3) node[v](4){};  
%\draw (1,3) node[v](5){};  
%\draw (6)--(1)--(2)--(3)--(4);
%\draw (3)--(5);
%\draw (7)--(1);
%\end{tikzpicture}
%\caption{A tree on $7$ vertices}\label{Fig:tree}
%\end{figure}

\end{Example}

\begin{Proposition}\label{prop:mintotdomintreediam2}
Let $T$ be a tree of order $n\geq 3$ with $\diam(T)=2$. Then 
$$\gamma_t(M(T))=n-1.$$
\end{Proposition}
\begin{proof} 
Since by assumption $\diam(T)=2$, then $T$ is isomorphic to $K_{1,n-1}$. This implies, by Proposition~\ref{prop:mintotdominstar}, that $\gamma_t(M(T))=n-1$.
\end{proof}

\begin{Remark}By the proof of Theorem~\ref{theomintotdomintreediam3}, differently from the case of domination (see \cite[Theorem 3.2]{dominmiddle}), $\gamma_t(M(G))=n-1$ does not implies that $G$ is isomorphic to $K_{1,n-1}$.
\end{Remark}

\section{Operations on graphs}

In this section, similarly to \cite{dominmiddle}, we study the total domination number of the middle graph of the corona, $2$-corona and join with $K_p$ of a graph.

\begin{Definition}
The \emph{corona} $G\circ K_1$ of a graph $G$ is the graph of order $2|V(G)|$ obtained from $G$ by adding a pendant edge to each vertex of $G$. %The \emph{$2$-corona} $G\circ P_2$ of $G$ is the graph of order $3|V(G)|$ obtained from $G$ by attaching a path of length $2$ to each vertex of $G$ so that the resulting paths are vertex-disjoint.
\end{Definition}

 \begin{Example} Consider the graph $P_3$, then the graph $P_3\circ K_1$ is the one in Figure~\ref{fig:pathscorona}.
 \end{Example}
 
 \begin{figure}[th]
\centering
\begin{tikzpicture}
\tikzstyle{v} = [circle, draw, inner sep=2pt, minimum size=3pt, fill=black]
\draw (-1,0) node[v](1){}; 
\draw (-1,-1) node[v](2){}; 
\draw (0,0) node[v](3){};
\draw (0,-1)  node[v](4){};
\draw (1,0)  node[v](5){};
\draw (1,-1)  node[v](6){};
\draw (1)--(3)--(5);
\draw (1)--(2);
\draw (3)--(4);
\draw (5)--(6);
\end{tikzpicture}
\caption{The graph $P_3\circ K_1$ }\label{fig:pathscorona}
\end{figure}
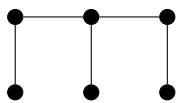

\begin{Theorem}\label{theo:mintotdomincorona}
For any connected graph $G$ of order $n\geq 2$, $$\gamma_t(M(G\circ K_1))=n+\gamma(M(G)).$$
\end{Theorem}
\begin{proof} 
To fix the notation, assume $V(G)=\{v_1,\dots, v_n\}$. Then $V(G\circ K_1)= \{v_{1},\dots, v_{2n}\}$ and $E(G\circ K_1)=\{v_1v_{n+1},\dots, v_nv_{2n} \}\cup E(G) $. Then $V(M(G\circ K_1))=V(G\circ K_1)\cup \mathcal{M}$, where $\mathcal{M}=\{ m_{i(n+i)}~|~1\leq i \leq n \}\cup \{ m_{ij}~|~v_iv_j\in  E(G)\}$.% and $E(M(G\circ K_1))=\{v_im_{i(n+i)},m_{i(n+i)}v_{n+i}~|~1\leq i \leq n \}\cup \{m_{ij}m_{ik}~|~v_iv_j\in  E(G)\}$. 

Let $S'$ be a minimal dominating set of $M(G)$. %By Lemma~\ref{lemma:totdominationisalledges}, we can assume that $S\subseteq E(G)$.
By construction, if we consider $S=S'\cup\{m_{i(n+i)}~|~1\leq i \leq n\}$, then $S$ is a total dominating set of $M(G\circ K_1)$ with $|S|=n+\gamma(M(G))$, and hence $\gamma_t(M(G\circ K_1))\le n+\gamma(M(G))$. 

On the other hand, let $S$ be a total dominating set of $M(G\circ K_1)$. By Lemma~\ref{lemma:totdominationisalledges}, we can assume that $S\subseteq \mathcal{M}$. 
Since $N_{M(G\circ K_1)}(v_{n+i})=\{m_{i(n+i)}\}$, for all $1\le i\le n$, then $m_{i(n+i)}\in S$, for all $1\le i\le n$. 
In addition, $N_{M(G\circ K_1)}(m_{i(n+i)})=\{v_i,v_{n+i}\}\cup N_{M(G)}(v_i)$, for all $1\le i\le n$, then $N_{M(G)}(v_i)\cap S\ne\emptyset$, for all $1\le i\le n$. This implies that $S\cap E(G)$ is a dominating set of $M(G)$ and hence $|S|\ge n+\gamma(M(G))$. This implies that $\gamma_t(M(G\circ K_1))\ge n+\gamma(M(G))$. 

\end{proof}

\begin{Definition} The \emph{$2$-corona} $G\circ P_2$ of a graph $G$ is the graph of order $3|V(G)|$ obtained from $G$ by attaching a path of length $2$ to each vertex of $G$ so that the resulting paths are vertex-disjoint.
\end{Definition}

 \begin{Example} Consider the graph $P_3$, then the graph $P_3\circ P_2$ is the one in Figure~\ref{fig:paths2corona}.
 \end{Example}
 
 \begin{figure}[th]
\centering
\begin{tikzpicture}
\tikzstyle{v} = [circle, draw, inner sep=2pt, minimum size=3pt, fill=black]
\draw (-1,0) node[v](1){}; 
\draw (-1,-0.5) node[v](2){}; 
\draw (-1,-1) node[v](3){}; 
\draw (0,0) node[v](4){};
\draw (0,-0.5) node[v](5){}; 
\draw (0,-1)  node[v](6){};
\draw (1,0)  node[v](7){};
\draw (1,-0.5) node[v](8){}; 
\draw (1,-1)  node[v](9){};
\draw (1)--(4)--(7);
\draw (1)--(2)--(3);
\draw (4)--(5)--(6);
\draw (7)--(8)--(9);
\end{tikzpicture}
\caption{The graph $P_3\circ P_2$ }\label{fig:paths2corona}
\end{figure}
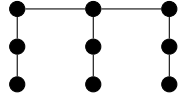

\begin{Theorem}\label{theomintotdomin2corona}
For any connected graph $G$ of order $n\geq 2$, $$\gamma_t(M(G\circ P_2))=2n.$$
\end{Theorem}
\begin{proof} 
To fix the notation, assume $V(G)=\{v_1,\dots, v_n\}$. Then $V(G\circ P_2)= \{v_{1},\dots, v_{3n}\}$ and $E(G\circ P_2)=\{v_iv_{n+i}, v_{n+i}v_{2n+i}~|~1\leq i \leq n \}\cup E(G) $. Then $V(M(G\circ P_2))=V(G\circ P_2)\cup \mathcal{M}$, where $\mathcal{M}=\{ m_{i(n+i)},m_{(n+i)(2n+i)}~|~1\leq i \leq n \}\cup \{ m_{ij}~|~v_iv_j\in  E(G)\}$.% and $E(M(G\circ P_2))=\{v_im_{i(n+i)},v_{n+i}m_{i(n+i)},v_{n+i}m_{(n+i)(2n+i)},v_{2n+i}m_{(n+i)(2n+i)}~|~1\leq i \leq n \}\cup \{m_{ij}m_{ik}~|~v_iv_j\in  E(G)\}$. 

Let $S$ be a total dominating set of $M(G\circ P_2)$. By Lemma~\ref{lemma:totdominationisalledges}, we can assume that $S\subseteq\mathcal{M}$. 
Since $N_{M(G\circ P_2)}(v_{2n+i})=\{m_{(n+i)(2n+i)}\}$, for every $1\leq i\leq n$, we have $m_{(n+i)(2n+i)}\in S$ for every $1\leq i\leq n$. In addition, $N_{M(G\circ P_2)}(m_{(n+i)(2n+i)})=\{m_{i(n+i)},v_{2n+i},v_{n+i}\}$, for every $1\leq i\leq n$, implies that $m_{i(n+i)}\in S$, for every $1\leq i\leq n$. This implies that $|S|\ge 2n$, and hence $\gamma_t(M(G\circ P_2))\ge2n$.

On the other hand, if we consider $S=\{ m_{i(n+i)},m_{(n+i)(2n+i)}~|~1\leq i \leq n \}$, then $S$ is a total dominating set of $M(G\circ P_2)$ with $|S|=2n$, and hence $\gamma_t(M(G\circ P_2))\le2n$.
\end{proof}

\begin{Definition}
The \emph{join} $G+ H$ of two graphs $G$ and $H$ is the graph with vertex set $V(G+H)=V(G)\cup V(H)$ and edge set $E(G+H)=E(G)\cup E(H)\cup \{vw~|~v\in V(G), w\in V(H)\}$.
%obtained by the disjoint union of $G$ and $H$ and joining each vertex of $G$ to all vertices of $H$.
\end{Definition}

\begin{Example} Consider the graphs $G=K_3$ and $H=P_2$, then graph $G+H$  is the one in Figure~\ref{fig:k3plusp2}.
 \end{Example}
 
 \begin{figure}[th]
\centering
\begin{tikzpicture}
\tikzstyle{v} = [circle, draw, inner sep=2pt, minimum size=3pt, fill=black]
\draw (-1,0.5) node[v](1){}; 
\draw (1,0.5) node[v](2){}; 
\draw (0,0) node[v](3){}; 
\draw (-0.5,-0.5) node[v](4){};
\draw (0.5,-0.5) node[v](5){}; 
\draw (1)--(2);
\draw (3)--(4)--(5)--(3);
\draw (1)--(3);
\draw (1)--(4);
\draw (1)--(5);
\draw (2)--(3);
\draw (2)--(4);
\draw (2)--(5);
\end{tikzpicture}
\caption{The graph $K_3+P_2$ }\label{fig:k3plusp2}
\end{figure}
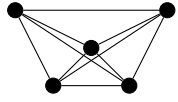

\begin{Theorem}\label{theo:mindominjoinpbig}
For any connected graph $G$ of order $n\geq 2$, 
$$\gamma_t(M(G+\overline{K_p}))= 
\begin{cases}
p & \text{ if $p\geq 2n$}\\ 
\lceil\frac{2(n+p)}{3}\rceil & \text{ if  $ \frac{n}{2}\le p\le 2n-1$.} 
\end{cases}$$
\end{Theorem}
\begin{proof}
To fix the notation, assume $V(G)=\{v_1,\dots,v_n\}$ and $V(\overline{K_p})=\{v_{n+1},\dots,v_{n+p}\}$. Then $V(M(G+\overline{K_p}))=V(G+\overline{K_p})\cup  \mathcal{M}_1\cup  \mathcal{M}_2$ where $\mathcal{M}_1= \{m_{ij}~|~v_iv_j\in  E(G)\}$ and $\mathcal{M}_2= \{m_{i(n+j)}~|~1\leq i \leq n, 1\leq j \leq p\}$.  

\textbf{Case {\boldmath$p\geq 2n$}.} Let $S$ be a total dominating set of $M(G+\overline{K_p})$. By Lemma~\ref{lemma:totdominationisalledges}, we can assume $S\subseteq \mathcal{M}_1\cup \mathcal{M}_2$. Since, if $j\ne k$, $N_{M(G+ \overline{K_p})}(v_{n+j})\cap N_{M(G+ \overline{K_p})}(v_{n+k})=\emptyset$, then for every $1\le j\le p$ there exists $1\le i\le n$ such that  $m_{i(n+j)}\in S$, and hence $|S|\ge p$. This implies that $\gamma_t(M(G+\overline{K_p}))\ge p$. On the other hand, if we consider $S=\{m_{i(n+i)},m_{i(2n+i)}~|~1\le i\le n\}\cup\{m_{1(3n+i)}~|~1\le i\le p-2n\}$, then $S$ is a total dominating set of $M(G+\overline{K_p})$ with $|S|=p$, and hence $\gamma_t(M(G+\overline{K_p}))\le p$.

\textbf{Case {\boldmath$p= 2n-1$}.} If we consider $S=\{m_{i(n+i)},m_{i(2n+i)}~|~1\le i\le n-1\}\cup\{m_{n(2n)},m_{n(2n-1)}\}$, then $S$ is a total dominating set of $M(G+\overline{K_{2n-1}})$ with $|S|=2n$, and hence $\gamma_t(M(G+\overline{K_{2n-1}}))\le 2n$. 
On the other hand, by Theorem~\ref{theo:lowerboundtotdomin}, $\gamma_t(M(G+\overline{K_{2n-1}}))\ge \lceil\frac{2(3n-1)}{3}\rceil =2n$, and hence $\gamma_t(M(G+\overline{K_{2n-1}}))= 2n=\lceil\frac{2(n+p)}{3}\rceil $.

\textbf{Case {\boldmath$n+3\le p\le 2n-2$}.} Assume that $p=n+k$ with $3\le k\le n-2$. The graph $G+\overline{K_p}$ has $k$ subgraphs isomorphic to $P_3$ and one subgraph isomorphic to $P_{2(n-k)}$ that are all disjoint. In fact, the $k$ subgraphs $(G+\overline{K_p})[v_1,v_{n+1},v_{2n+1}], \dots, (G+\overline{K_p})[v_k, v_{n+k},v_{2n+k}]$ are all isomorphic to $P_3$ and the subgraph $(G+\overline{K_p})[v_{k+1},\dots, v_{n},v_{n+k+1},\dots, v_{2n}]$ has a subgraph isomorphic to $P_{2(n-k)}$.
By Proposition~\ref{prop:mintotdominpath}, this implies that $\gamma_t(M(G+\overline{K_p}))\le 2k+\lceil\frac{2(2(n-k))}{3}\rceil=\lceil\frac{2(n+p)}{3}\rceil$. By Theorem~\ref{theo:lowerboundtotdomin}, we obtain the desired equality.

\textbf{Case {\boldmath$p= n+2$}.} If $n\equiv 0 \mod 3$, consider $$S=\{m_{1(n+1)},m_{1(n+2)},m_{2(n+3)},m_{3(n+3)},\dots, m_{(n-1)(2n)}, m_{n(2n)}, m_{n(2n+1)}, m_{n(2n+2)}\}.$$ Then $S$ is a total dominating set of $M(G+\overline{K_p})$ with $|S|=\lceil\frac{2(n+p)}{3}\rceil$. If $n\equiv 1 \mod 3$, consider $$S=\{m_{1(n+1)},m_{1(n+2)},m_{2(n+3)},m_{3(n+3)},\dots, m_{n(2n)}, m_{n(2n+1)}, m_{n(2n+2)}\}.$$ Then $S$ is a total dominating set of $M(G+\overline{K_p})$ with $|S|=\lceil\frac{2(n+p)}{3}\rceil$. If $n\equiv 2 \mod 3$, consider $$S=\{m_{1(n+1)},m_{1(n+2)},m_{2(n+3)},m_{3(n+3)},\dots, m_{n(2n+1)}, m_{n(2n+2)}\}.$$ Then $S$ is a total dominating set of $M(G+\overline{K_p})$ with $|S|=\lceil\frac{2(n+p)}{3}\rceil$. This implies that $\gamma_t(M(G+\overline{K_p}))\le \lceil\frac{2(n+p)}{3}\rceil$. By Theorem~\ref{theo:lowerboundtotdomin}, we then obtain that $\gamma_t(M(G+\overline{K_p}))= \lceil\frac{2(n+p)}{3}\rceil$.

\textbf{Case {\boldmath$n-1\le p\le n+1$}.} If $p=n-1$, then the graph $G+\overline{K_p}$ contains the path $P: v_1v_{n+1}v_2v_{n+2}\cdots v_{n+p}v_n$. If $p=n$, then $G+\overline{K_p}$ contains the path $P': v_1v_{n+1}v_2v_{n+2}\cdots v_{n+p-1}v_nv_{n+p}$. If $p=n+1$, then $G+\overline{K_p}$ contains the path $P'': v_{n+1}v_1v_{n+2}v_2v_{n+3}\cdots v_{n+p-1}v_nv_{n+p}$. Since the paths $P, P'$ and $P''$ are all isomorphic to $P_{n+p}$, we can apply Theorem~\ref{prop:graphwithapath}, and obtain that $\gamma_t(M(G+\overline{K_p}))= \lceil\frac{2(n+p)}{3}\rceil$.%, when $n-1\le p\le n+1$.

\textbf{Case {\boldmath$\frac{n}{2}\le p\le n-2$}.} Assume that $p=n-k$ with $2\le k\le \frac{n}{2}$. If $n$ is even and $p=\frac{n}{2}$ (or equivalently $k=\frac{n}{2}$), then the set $S=\{m_{i(n+i)}, m_{(i+\frac{n}{2})(n+i)}~|~1\le i\le \frac{n}{2}\}$ is a total dominating set of $M(G+\overline{K_p})$ with $|S|=n=\lceil\frac{2(n+p)}{3}\rceil$. This implies that $\gamma_t(M(G+\overline{K_p}))\le \lceil\frac{2(n+p)}{3}\rceil$, and by Theorem~\ref{theo:lowerboundtotdomin}, we obtain the desired equality.

Assume that $2\le k\le \frac{n}{2}-1$. The graph $G+\overline{K_p}$ has $k$ subgraphs isomorphic to $P_3$ and one subgraph isomorphic to $P_{2(n-2k)}$ that are all disjoint. In fact, the $k$ induced subgraphs $(G+\overline{K_p})[v_1,v_{n+1},v_{k+1}], \dots, (G+\overline{K_p})[v_k, v_{n+k},v_{2k}]$ are all isomorphic to $P_3$ and the induced subgraph $(G+\overline{K_p})[v_{2k+1},\dots, v_{n},v_{n+k+1},\dots, v_{2n-k}]$ has a subgraph isomorphic to $P_{2(n-2k)}$.
By Proposition~\ref{prop:mintotdominpath}, this implies that $\gamma_t(M(G+\overline{K_p}))\le 2k+\lceil\frac{2(2(n-2k))}{3}\rceil=\lceil\frac{2(n+p)}{3}\rceil$. By Theorem~\ref{theo:lowerboundtotdomin}, we obtain the desired equality.
\end{proof}

Similarly to \cite{dominmiddle}, when $p$ is small relatively to $n$, $\gamma_t(M(G+\overline{K_p}))$ is strongly related to $\gamma_t(M(G))$.

\begin{Theorem}\label{theo:mindominjoinverysmallrange}
For any connected graph $G$ of order $n\geq 2$ and any integer $1\le p\le \frac{n}{2}-1$, 
$$\lceil\frac{2(n+p)}{3}\rceil\le \gamma_t(M(G+\overline{K_p}))\le$$ $$2p+\min\{\gamma_t(M(G[A]))~|~A\subseteq V(G), $$ $$|A|=n-2p, G[A] \text{ has no isolated vertices}\}.$$
\end{Theorem}
\begin{proof} To fix the notation, assume $V(G)=\{v_1,\dots,v_n\}$ and $V(\overline{K_p})=\{v_{n+1},\dots,v_{n+p}\}$. Then $V(M(G+\overline{K_{p}}))=V(G+\overline{K_{p}})\cup  \mathcal{M}_1\cup  \mathcal{M}_2$ where $\mathcal{M}_1= \{m_{ij}~|~v_iv_j\in  E(G)\}$ and $\mathcal{M}_2= \{m_{i(n+j)}~|~1\leq i \leq n, 1\leq j \leq p\}$.

By Theorem~\ref{theo:lowerboundtotdomin}, we obtain the first inequality. Let now $A\subseteq V(G)$ be such that $|A|=n-2p$ and $G[A]$ has no isolated vertices. Without loss of generalities, we can assume that $A=\{v_{2p+1},\dots, v_{n}\}$. Consider $S'$ be a minimal total dominating set of $M(G[A])$, then $S=S'\cup \{m_{i(n+i)},m_{(p+i)(n+i)}~|~1\le i\le p\}$ is a total dominating set of $M(G+\overline{K_p})$. Since this arguments works for every $A\subseteq V(G)$ such that $|A|=n-2p$ and $G[A]$ has no isolated vertices, we obtain the second inequality.

\end{proof}

If we apply Lemma~\ref{lemma:totdominationdeletionvertex} to the graph $G+\overline{K_1}$, we obtain the following result.
\begin{Lemma}\label{lemma:plusonevertextotaldomin} Let $G$ be a graph of order $n\geq 2$ with no isolated vertices. Then  
$$\gamma_t(M(G))\le \gamma_t(M(G+\overline{K_1}))\le \gamma_t(M(G))+1.$$
\end{Lemma}

Notice that both inequalities described in Lemma~\ref{lemma:plusonevertextotaldomin} are sharp as the following examples show.
\begin{Example} Consider the graph $G=C_5$. Then $G+\overline{K_1}$ is isomorphic to $W_6$. This implies that by Corollary~\ref{corol:mintotdominfamily}, $\gamma_t(M(G))=4=\gamma_t(M(G+\overline{K_1}))$.
\end{Example}

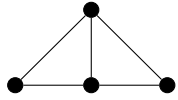
\begin{figure}[th]
\centering
\begin{tikzpicture}
\tikzstyle{v} = [circle, draw, inner sep=2pt, minimum size=3pt, fill=black]
\draw (0,0) node[v](1){}; 
\draw (-1,0) node[v](2){}; 
\draw (1,0) node[v](3){};
\draw (0,1) node[v](4){}; 
\draw (1)--(2)--(4)--(1);
\draw (4)--(3)--(1);
\end{tikzpicture}
\caption{The graph $P_3+\overline{K_1}$}\label{Fig:sumwithK1}
\end{figure}

\begin{Example} Consider the graph $G=P_3$. Then $G+\overline{K_1}$ is the graph in Figure~\ref{Fig:sumwithK1}. By Proposition~\ref{prop:mintotdominpath} and Theorem~\ref{prop:graphwithapath}, we have that $\gamma_t(M(P_3))=2$ and $\gamma_t(M(P_3+\overline{K_1}))=3$.
\end{Example}

\begin{Proposition}\label{prop:joinstargraphk1}
For any star graph $K_{1,n}$ on $n+1$ vertices, with $n\ge 4$, we have  
$$\gamma_t(M(K_{1,n}+\overline{K_1}))=n.$$
\end{Proposition}
\begin{proof} To fix the notation, assume $V(K_{1,n})=\{v_0,v_1,\dots, v_n\}$, $V(\overline{K_1})=\{v_{n+1}\}$ and $E(K_{1,n})=\{v_0v_1,v_0v_2,\dots, v_0v_n\}$. Then $V(M(K_{1,n}+\overline{K_1}))=V(K_{1,n})\cup \mathcal{M}$, where $\mathcal{M}=\{ m_i~|~1\leq i \leq n \}\cup\{m_{i(n+1)}~|~0\le i\le n\}$.

By Proposition~\ref{prop:mintotdominstar} and Lemma~\ref{lemma:plusonevertextotaldomin}, $\gamma_t(M(K_{1,n}+\overline{K_1}))\ge n$. On the other hand, consider $S=\{m_i~|~1\le i\le n-2\}\cup\{m_{(n-1)(n+1)},m_{n(n+1)}\}$ is a total dominating set of $M(K_{1,n}+\overline{K_1})$ with $|S|=n$, and hence $\gamma_t(M(K_{1,n}+\overline{K_1}))\le n$.
\end{proof}

\begin{Remark} Proposition~\ref{prop:joinstargraphk1} shows that the upper bound of Theorem~\ref{theo:mindominjoinverysmallrange} is sharp. In fact, if $A\subseteq V(K_{1,n})$ with $|A|=n-2$ and $G[A]$ has no isolated vertices, then $G[A]$ is isomorphic to $K_{1,n-2}$, and hence by Proposition~\ref{prop:mintotdominstar}, $\gamma_t(M(G[A]))=n-2$.
\end{Remark}

\begin{Proposition}\label{prop:withpathmintotdominsumKp} Let $G$ be a graph of order $n\ge 2$ and $1\le p\le \frac{n}{2}-1$. If $G$ has a subgraph isomorphic to a path graph $P_n$, then 
$$\gamma_t(M(G+\overline{K_p}))=\lceil\frac{2(n+p)}{3}\rceil.$$
\end{Proposition}
\begin{proof} Under our assumption, the graph $G+\overline{K_p}$ contains a subgraph isomorphic to $P_{n+p}$. By Theorem~\ref{prop:graphwithapath}, this implies that $\gamma_t(M(G+\overline{K_p}))=\lceil\frac{2(n+p)}{3}\rceil.$
\end{proof}

As a direct consequence of Proposition~\ref{prop:withpathmintotdominsumKp}, we obtain the following result.

\begin{Corollary}\label{cor:knownfamilymintotdominsumKp} Let $G$ be a graph of order $n\ge 2$ and $1\le p\le \frac{n}{2}-1$. If $G$ is isomorphic to a path graph $P_n$, or a cycle graph $C_n$, or a wheel graph $W_n$, or a complete graph $K_n$, then 
$$\gamma_t(M(G+\overline{K_p}))=\lceil\frac{2(n+p)}{3}\rceil.$$
\end{Corollary}

\section{Nordhaus-Gaddum relations}

In \cite{Nordhaus}, Nordhaus and Gaddum gave a lower bound and an upper bound, in terms of the order of the graph, on the sum and the product of the chromatic number of a graph and its complement. 
Since then, lower and upper bounds on the sum and the product of many other graph invariants, like domination and total domination numbers, have been proposed by several authors. See \cite{Aouchiche13} for a survey on the subject.

\begin{Theorem}\label{theo:Nordgadtotdomin}
Let $G$ be a graph on $n\ge2$ vertices. Assume that the graphs $G$ and $\overline{G}$ have no isolated vertices and no components isomorphic to $K_2$. Then
$$2(n-1)\ge \gamma_t(M(G))+\gamma_t(M(\overline{G}))\ge 2\lceil\frac{2n}{3}\rceil $$
and
$$(n-1)^2\ge \gamma_t(M(G))\cdot\gamma_t(M(\overline{G}))\ge (\lceil\frac{2n}{3}\rceil)^2. $$
\end{Theorem}
\begin{proof} 
By applying Theorem~\ref{theo:lowerboundtotdomin} to each component of $G$ and $\overline{G}$, we obtain that $n-1\ge\gamma_t(M(G))\ge \lceil\frac{2n}{3}\rceil $ and $n-1\ge\gamma_t(M(\overline{G}))\ge \lceil\frac{2n}{3}\rceil$.
\end{proof}

\begin{Remark} If in Theorem~\ref{theo:Nordgadtotdomin} we allow $G$ or $\overline{G}$ to have components isomorphic to $K_2$, then the described upper bounds might not work. To see this it is enough to consider the graph $C_4$. In fact, $\overline{C_4}$ consists of two copies of $K_2$, and then $\gamma_t(M(C_4))=3$ and $\gamma_t(M(\overline{C_4}))=4$.
\end{Remark}

Notice that all the inequalities of Theorem~\ref{theo:Nordgadtotdomin} are sharp, in fact we have the following example.

\begin{Example}
Consider the graph $P_4$, then by Proposition~\ref{prop:mintotdominpath}, we have $\gamma_t(M(P_4))=3$. On the other hand, $\overline{P_4}$ is isomorphic to $P_4$, and hence $\gamma_t(M(\overline{P_4}))=3$. Since $n=4$, then $6=\gamma_t(M(P_4))+\gamma_t(M(\overline{P_4}))=2(n-1)=2\lceil\frac{2n}{3}\rceil$, and $9=\gamma_t(M(P_4))\cdot\gamma_t(M(\overline{P_4}))=(n-1)^2=(\lceil\frac{2n}{3}\rceil)^2$.
 \end{Example}
\paragraph{\textbf{Acknowledgements}} During the preparation of this article the fourth author was supported by JSPS Grant-in-Aid for Early-Career Scientists (19K14493).

%\newpage

\end{document}